\documentclass{amsart} 
\usepackage{graphicx} 

\usepackage{hyperref}
\usepackage{float}
\hypersetup{
    colorlinks=true,
    citecolor=blue,
    linkcolor=blue,
    filecolor=magenta,      
    urlcolor=cyan,
    pdftitle={Braid closure union braid axis is minimal},
    pdfpagemode=FullScreen,
    }

\usepackage[margin=1in]{geometry}

\newcommand{\hatfl}{\widehat{HFL}}

\usepackage{amsmath,amssymb,amsthm, mathrsfs}
\usepackage{dsfont}
\usepackage[dvipsnames]{xcolor}
\usepackage{tikz}
    \usetikzlibrary{cd}
\usepackage{tikz-cd}
\allowdisplaybreaks
\usepackage{standalone}
\usepackage{comment}
\usepackage{wasysym} 
\usepackage{mathtools} 
\usepackage{bbm} 
\usepackage{stmaryrd} 
\usepackage{enumitem}

\theoremstyle{definition}
\newtheorem{theorem}{{Theorem}}[subsection]

\newtheorem{lemma}[theorem]{{Lemma}}
\newtheorem{proposition}[theorem]{{Proposition}}
\newtheorem{definition}[theorem]{{Definition}}
\newtheorem{corollary}[theorem]{{Corollary}}

\newtheorem{remark}[theorem]{{Remark}}

\newcommand{\bea}{\begin{enumerate}[label=(\alph*)]}
\newcommand{\ee}{\end{enumerate}}

\newcommand{\Q}{\mathbb{Q}}
\newcommand{\R}{\mathbb{R}}
\newcommand{\Z}{\mathbb{Z}}
\newcommand{\F}{\mathbb{F}} 

\renewcommand{\emptyset}{\varnothing} 

\newcommand{\rank}{\mathrm{rank}~}

\newcommand{\id}{\mathrm{id}}

\title{Braid closure union braid axis is ribbon concordance minimal}
\author[B.~Daniels]{Benjamin Daniels}
\address{Department of Mathematics, University of California Berkeley, Berkeley, CA 94720-3840, U.S.A.
}
\email{\href{mailto:benjamin_daniels@berkeley.edu}{benjamin\_daniels@berkeley.edu}}
\date{\today}

\begin{document}

\begin{abstract}
    We show that a ribbon concordance minimal fibered knot $K$ in $S^3$ can generate ribbon concordance minimal links by the addition of any braid closure in $S^3 \setminus K$.
    As a corollary, we show that any link $L$ in $S^3$ may be made ribbon concordance minimal by adding a single unknot linked with $L$.
    Our proofs use link Floer homology together with classical techniques.
\end{abstract}

\maketitle

\section{Introduction}
\label{sec:introduction}

A \textit{ribbon concordance} from a knot $K_0$ to a knot $K_1$ is an embedded annulus $R \subset S^3 \times I$ ($I$ is the unit interval) such that $R \cap \{0\} = K_0$, $R \cap \{1\} = K_1$, and the projection $S^3 \times I \to I$ restricts to a Morse function on $R$ with no index $2$ critical points.
In this case, we say that $K_0$ is \textit{ribbon concordant} to $K_1$, and write $K_0 \leq K_1$.
Introduced by Gordon in \cite{Gordon-ribcon}, the existence of a ribbon concordance $R$ from $K_0$ to $K_1$ is intuitively expected to bound the complexity of $K_0$ in terms of that of $K_1$.
This intuition has been formalized and corroborated by a suite of conjectures and results; see, for instance, \cite{Gordon-ribcon}, \cite{Gilmer-alexanderpoly}, \cite{Zemke-knotfloerribcon}, \cite{LevineZemke-khribcon}, \cite{BaldwinSivek-fiberedribconI}, \cite{BaldwinHanselmanSivek-fiberedribconII}, \cite{AgolRen-fiberedribcon}, and \cite{Sun-fiberedclasses}.
In particular, Agol proved in \cite{Agol-partialordering} that ribbon concordance provides a partial ordering on the set of knots in $S^3$ (justifying the notation ``$\leq$'').

In the past few years, authors have been considering natural generalizations of ribbon concordance for knots, notably \textit{ribbon $\Q$-homology cobordisms} for $3$-manifolds \cite{DaLidVVShea-ribhomcob}.
This paper will focus on ribbon concordance for links in $S^3$.
A \textit{ribbon concordance of links} from $L_0$ to $L_1$ is a disjoint union of ribbon concordances from the components of $L_0$ to the components of $L_1$.
Specifically, we will study \textit{ribbon concordance minimal links}.
A link $L$ is ribbon concordance minimal if the existence of a ribbon concordance from $L_0$ to $L$ forces $L_0 \cong L$.
Dunkerley found infinite families of examples of ribbon concordance minimal links in Theorems 1.9, 1.11, and 1.12 of \cite{Dunkerley-linkribcon}. 
We will contribute to this search with the following result:

\begin{theorem}
    \label{thm:braidclosureunionbraidaxisisminimal}
    A fibered knot $K$ is ribbon concordance minimal if and only if for every braid closure $\hat{\beta}$ in $S^3 \setminus K$, the link $K \cup \hat{\beta}$ is ribbon concordance minimal.
\end{theorem}

Examples of fibered knots which are ribbon concordance minimal are in relative abundance.
For instance, $K$ could be a fibered strongly quasi-positive knot (see Remark 3.4 of \cite{Boninger-positiveknots}).
Restricting $K$ to be the unknot yields the following corollary:

\begin{corollary}
    \label{cor:addingbraidaxistomakemin}
    One can add to any link $L$ in $S^3$ a single unknotted component $U$, linked nontrivially with $L$, whereby the resulting link $L \cup U$ is ribbon concordance minimal.
\end{corollary}

\begin{proof}
    There is always an unknotted axis $U$ such that $L$ is a braid closure with respect to $U$ by Alexander's theorem, concluding the proof by Theorem \ref{thm:braidclosureunionbraidaxisisminimal}.
\end{proof}

Corollary \ref{cor:addingbraidaxistomakemin} gives a positive answer to a strong form of Conjecture 7.6 in \cite{Dunkerley-linkribcon}.

\section{Preliminaries and notation}
\label{sec:preliminaries}

We begin by introducing the ideas we will make frequent use of in the proofs to follow.

\subsection{Topological background}
\label{subsec:topologicalbackground}

Recall that an (oriented) knot $K$ in $S^3$ is \textit{fibered} if $S^3 \setminus K$ is fibered over $S^1$ with each fiber being a Seifert surface for $K$.
In this case, an oriented link $B$ is said to be a \textit{braid closure of index $\alpha$} in $S^3 \setminus K$ if

\begin{itemize}
    \item $B$ can be isotoped so that it is always transverse to our fibration of $S^3 \setminus K$,
    \item The orientation on $B$ is such that $\Sigma \cdot B$ is maximized, and with this orientation we have an equality $\Sigma \cdot B = \alpha$.
\end{itemize}

When $K$ is considered as a component of a link $L$ such that $L \setminus K$ is a braid closure in the complement of $K$, we will sometimes refer to $K$ as a \textit{braid axis}.

Moving to ribbon concordance, we will write $L_0 \leq L_1$ to mean that there is a ribbon concordance from $L_0$ to $L_1$.
If we wish to name this ribbon concordance, we will reserve the notation $R: L_0 \to L_1$ to mean that $R$ is a ribbon concordance from $L_0$ to $L_1$.
Link groups provide powerful obstructions to ribbon concordance by the following observation of Gordon, later generalized by other authors.
Here, and going forward, $\nu X$ denotes a neighborhood of $X$.

\begin{lemma}[Lemma 3.1 of \cite{Gordon-ribcon}, Proposition 2.1 of \cite{DaLidVVShea-ribhomcob}, Lemma 4.5 of \cite{Dunkerley-linkribcon}]
    \label{lma:pi1injsurj}
    Suppose we have a ribbon concordance of links $R: L_0 \to L_1$.
    The inclusion maps $\iota_0,\iota_1: S^3 \setminus \nu L_0, S^3 \setminus \nu L_1 \to (S^3 \times I) \setminus \nu R$ induce an injection and a surjection on $\pi_1$ respectively.
\end{lemma}

Lastly, we will recall the definition of the \textit{Thurston norm} in our setting.
For a compact surface $\Sigma = \bigcup_{i} \Sigma_i$ (possibly with boundary) embedded in a $3$-manifold $Y$, we can define the \textit{complexity} of $\Sigma$ by 
    \[\chi_-(\Sigma) = \sum_{\{i : \chi(\Sigma_i) \leq 0\}} -\chi(\Sigma_i).\]
Given an $n$-component link $L$, we will define
    \[x: H_2(S^3, L; \Z) \to \Z, \quad x(h) = \min_{\text{embedded surfaces $\Sigma \subset S^3 \setminus \nu L$, $[\Sigma] = h$}} \chi_-(\Sigma).\]
Thurston notes that $x$ may be extended to a function $H_2(S^3, L; \R) \to \R$, and to the \textit{dual Thurston norm} by the isomorphism $H^1(S^3 \setminus \nu L; \R) \cong H_2(S^3, L; \R)$.
We then obtain the \textit{dual Thurston norm ball} $B_{x^*} = \{h \in H^1(S^3 \setminus \nu L; \R) : x^*(h) \leq 1\}$.
It is proven in \cite{Thurston-norm} that $B_{x^*}$ is a convex polytope with remarkable properties.

\subsection{Link Floer homology}
\label{subsec:linkfloerhomology}

One of the main characters in our proof of Theorem \ref{thm:braidclosureunionbraidaxisisminimal} is \textit{link Floer homology}, introduced by Ozsv\'ath and Szab\'o in \cite{OzsvathSzabo-linkfloer}.
Link Floer homology is an invariant which assigns to an $n$-component link $L$ in $S^3$ an $n+1$-graded $\F_2$ vector space $\hatfl(L)$.
The first grading is the Maslov grading, which will not play a role in this paper, while the latter $n$ gradings $a_1, ..., a_n$ compile into a single vector-valued grading called the \textit{multi-Alexander grading}.
We then write
    \[\hatfl(L) = \bigoplus_{a_1, ..., a_n}\hatfl(L, a_1, ..., a_n).\]
In particular, we can think of $a_i$ as corresponding to the $i$th component $L_i$ of $L$, in which case $a_i$ satisfies
    \[2a_i + \ell(L_i, L \setminus L_i) \in 2\Z.\]
We will set $\hatfl(L, a_i)$ to mean the piece of $\hatfl(L)$ where the $i$th Alexander grading is $a_i$, and let 
    \[m_{L_i} = \max_{a_i} \{\hatfl(L, a_i) \neq 0\}.\]
For our purposes, we will identify $\R^n$ with $H_1(S^3 \setminus \nu L; \R)$, and view a multi-Alexander grading $(a_1, ..., a_n)$ as the element 
    \[\sum_i a_i [\mu_i] \in H_1(S^3 \setminus \nu L; \R),\]
where $\mu_i$ denotes the meridian of $L_i$.
This means that for an element $h \in H^1(S^3 \setminus \nu L; \R)$, we can speak of the Kronecker pairing $\langle h, (a_1, ..., a_n)\rangle \in \R$.

Link Floer homology enjoys a number of structural properties which make it conducive to proving Theorem \ref{thm:braidclosureunionbraidaxisisminimal}.
It is functorial with respect to \textit{decorated cobordisms} (see \cite{Zemke-knotfloerribcon} for terminology).
In turn, $\hatfl$ obstructs ribbon concordance by a theorem of Zemke, with the link Floer adaptation noted by Wang \cite{Wang-coscrosconsplitlinks}, Guth \cite{Guth-ribhomcoblinkfloer}, and Dunkerley \cite{Dunkerley-linkribcon}:

\begin{theorem}[Theorem 1.1 of \cite{Zemke-knotfloerribcon}]
    \label{thm:linkfloermonotonicity}
    Let $C$ be a decorated cobordism obtained from a ribbon concordance $R: L_0 \to L_1$ of links.
    Suppose that each component of $C$ is decorated with a single $\textbf{w}$ arc and a single $\textbf{z}$ arc.
    Then $\hatfl(C): \hatfl(L_0) \to \hatfl(L_1)$ is a grading-preserving injection.
\end{theorem}

Recall that the idea of this proof is to show that if we take $\bar{C}$, the decorated cobordism obtained by reversing $C$, we have $\hatfl(\bar{C}) \circ \hatfl(C) = \id$. 
The mechanics of what makes a proof like this work for a suitable link homology theory is explored in \cite{Kang-linkhomologytheoriesribcon}.

We can also associate a polytope to $\hatfl(L)$, called the \textit{link Floer polytope}, which is defined by 
    \[P(L) = \mathrm{conv}\{s \in \R^n : \hatfl(L, s) \neq 0\}.\]
It was proven in \cite{OzsvathSzabo-thurstonnorm} that if $L$ has no split unknotted component, then 
    \[2P(L) = B_{x^*} + [-1,1]^n\]
(written in this form in Theorem 4.2 of \cite{Kim-secondsmallestfloer}).
Here, $+$ means Minkowski sum.

\section{Proof of the main result and some remarks}
\label{sec:proof}

We begin by proving a strengthened version of a braid axis detection theorem of Martin \cite{Martin-T26detection}, relayed to the author by Fraser Binns.

\begin{proposition}
    \label{prop:braidaxisdetect}
    Let $L$ be an oriented $n$-component link in $S^3$ with a fibered component $K$ such that every component of $L \setminus K$ has nonzero geometric linking with $K$.
    Then the following hold:

    \begin{itemize}
        \item[(a).] There is an equality 
            \[m_K = \frac{1}{2}\min_{\text{Seifert surface $\Sigma$ for $K$}} (2g(\Sigma) + |\Sigma \cap (L \setminus K)|).\]
        \item[(b).] We have 
            \[\rank \hatfl(L, m_K) \geq 2^{n-1}\]
        with equality if and only if $L \setminus K$ is a braid closure in the complement of $K$.
        \item[(c).] In the case that $L \setminus K$ is a braid closure of index $\alpha$ in the complement of $K$, then (a). may be specialized to the equality $m_K = g(K) + \alpha/2$.
    \end{itemize}

\end{proposition}

\begin{proof}
    We prove (a) first. 
    Up to reordering components, we can think of $K$ as our $n$th component.
    By Theorem 1.1 of \cite{OzsvathSzabo-thurstonnorm} we have an equality
        \[x(PD[h]) + \sum_{i=1}^\ell|\langle h, \mu_i\rangle| = 2\max_{\{s | \hatfl(L, s) \neq 0\}}|\langle h, s\rangle|\]
    for any $h \in H^1(S^3 \setminus \nu L; \R)$, where $\mu_i$ is the meridian of the $i$th component.
    Identifying $H^1(S^3 \setminus \nu L; \R)$ with $\hom_{\R}(H_1(S^3 \setminus \nu L; \R), \R)$, we can choose $\xi$ to be the cohomology class sending the element $\mu_n\in H_1(S^3 \setminus \nu L; \R)$ to $1$ and all other meridians to zero. 
    Then our equation reduces to
        \[x(PD[\xi]) + 1 = 2\max_{\{s | \hatfl(L, s) \neq 0\}}|\langle \xi, s\rangle| = 2m_K.\]
    The class $PD[\xi]$ is represented by the class of a surface in $H_2(S^3 \setminus \nu L; \R)$ with a boundary component on $\partial \nu K$ intersecting $\mu_n$ once, and the other boundary components representing $(t, 0)$-curves on components of $\partial \nu (L \setminus K)$.
    We can cap off these curves in $S^3 \setminus \nu(L \setminus K)$ to see a surface with a single boundary component on $\partial \nu K$; this component must represent a longitude on $\partial \nu K$.
    It follows that an embedded surface in the class of $PD[\xi]$ is obtained by choosing a Seifert surface $\Sigma$ for $K$ and removing a small disk around every point in $\Sigma \cap (L \setminus K)$.
    The total Euler characteristic of such a surface is given by 
        \[1 - 2g(\Sigma) - |\Sigma \cap (L \setminus K)|,\]
    so that we can view $x(PD[\xi]) + 1$ as the minimum over all embedded surfaces $\Sigma$ in $S^3 \setminus \nu K$ (and bounding a longitude on $\partial \nu K$) of the quantity $2g(\Sigma) + |\Sigma \cap (L \setminus K)|$.

    Next, we tackle the general inequality of (b).
    From Theorem 1.1 of \cite{OzsvathSzabo-thurstonnorm}, we can express $2P(L)$ as the Minkowski sum 
        \[2P(L) = B_{x^*} + [-1,1]^n\]
    where $B_{x^*}$ is the dual Thurston norm unit ball.
    Choose coordinates $(a_1, ..., a_n)$ on $\R^n$, and intersect $2P(L)$ with the coordinate hyperplane $\{a_n = 2m_K\}$.
    This produces the (nonempty) convex polytope
        \[\{(b_1, ..., b_n) + (z_1, ..., z_n) \in B_{x^*} + [-1,1]^n : b_n + z_n = 2m_K\}.\]
    By varying the coordinates $z_1, ..., z_{n-1}$, we explicitly see that this slice of $2P(L)$ must be the Minkowski sum of some convex polytope and $[-1,1]^{n-1}$.
    From Lemma 4.3 of \cite{Kim-secondsmallestfloer}, we see that $2P(L) \cap \{a_n = 2m_K\}$ has at least $2^{n-1}$ vertices.
    Since we chose $\{a_n = 2m_K\}$ to be our coordinate hyperplane, we also see that 
        \[2P(L) \cap \{a_n > 2m_K\} = \emptyset,\]
    making $\{a_n \leq 2m_K\}$ a supporting hyperplane for $2P(L)$.
    The vertices of $2P(L) \cap \{a_n = 2m_K\}$ are then vertices of $2P(L)$.
    As vertices of $2P(L)$ correspond to classes in $\hatfl(L)$, this produces the inequality
        \[\rank \hatfl(L, m_K) \geq 2^{n-1}.\]
    The statement that $\rank \hatfl(L, m_K) = 2^{n-1}$ if and only if $L \setminus K$ is a braid closure in the complement of $K$ is precisely Proposition 1 of \cite{Martin-T26detection}.

    It remains to show (c).
    By assumption, $\ell(L \setminus K, K) = \alpha$.
    Choose a Seifert surface $\Sigma_K$ which is a fiber surface for the fibration of $S^3 \setminus \nu K$.
    Then $\Sigma_K$ is minimal genus (see, e.g., Corollary 2 of \cite{Thurston-norm}).
    By assumption, there is an isotopy of $L \setminus K$ so that it is always transverse to the fibration of $S^3 \setminus \nu K$, meaning that $|\Sigma_K \cap (L \setminus K)| = \alpha$.
    For any other Seifert surface $\Sigma$ for $K$, we have $g(\Sigma) \geq g(\Sigma_K)$, and moreover $|\Sigma \cap (L \setminus K)| \geq \Sigma \cdot (L \setminus K) = \alpha$.
    The claim follows from our proof of (a). 
\end{proof}

\begin{remark}
    \label{rem:generalityofthebraidaxis}
    The eagle-eyed reader will notice that we did not use the fact that $K$ was fibered in our proof of (a). nor in our proof of the inequality in (b)., and the assumption that every component of $L \setminus K$ has nonzero geometric linking with $K$ may be relaxed to the assumptions that $L$ has no trivial components (see Theorem 1.1 of \cite{OzsvathSzabo-thurstonnorm}).
\end{remark}

We can now prove the first major component of Theorem \ref{thm:braidclosureunionbraidaxisisminimal}.
One can view this as a version of the statement that the property of being a braid axis is downward closed under ribbon concordance.

\begin{lemma}
    \label{lma:braidclosurepreservation}
    Let $L_1$ be an oriented $n$-component link in $S^3$ with a component $K_1$ such that $K_1$ is fibered, ribbon concordance minimal, and $L_1 \setminus K_1$ is a braid closure of index $\alpha$ in the complement of $K_1$.
    Suppose that $L_0$ is ribbon concordant to $L_1$, and that the component $K_0$ of $L_0$ corresponds to $K_1$ under this ribbon concordance.
    Then $K_0$ is fibered and $L_0 \setminus K_0$ is a braid closure of index $\alpha$ in the complement of $K_0$.
\end{lemma}

\begin{proof}
    By assumption, $\ell(L_1 \setminus K_1, K_1) = \alpha$.
    As linking number is a concordance invariant \cite{Milnor-linkgroups}, it follows that $\ell(L_0 \setminus K_0, K_0) = \alpha$ as well.
    Since $K_1$ is ribbon concordance minimal, $K_0 \cong K_1$ as knots in $S^3$.
    Thus, $K_0$ is fibered, and moreover $g(K_0) = g(K_1)$.
    From Proposition \ref{prop:braidaxisdetect} (b). and (c). we have equalities 
        \[m_{K_1} = g(K_1) + \alpha/2, \quad \rank \hatfl(L, m_{K_1}) = 2^{n-1}.\]
    We also have

    \begin{equation*}
        \begin{split}
            g(K_1) + \alpha/2 & = m_{K_1}\\
            & \geq m_{K_0} \quad \quad \text{Theorem \ref{thm:linkfloermonotonicity}}\\
            & = \frac{1}{2}\min_{\text{Seifert surface $\Sigma$ for $K_0$}} (2g(\Sigma) + |\Sigma \cap (L_0 \setminus K_0)| \quad \quad \text{Proposition \ref{prop:braidaxisdetect}}\\
            & \geq \min_{\text{Seifert surface $\Sigma$ for $K_0$}} g(\Sigma) + \frac{\Sigma \cdot (L_0 \setminus K_0)}{2}\\
            & = \min_{\text{Seifert surface $\Sigma$ for $K_0$}} g(\Sigma) + \frac{\ell(L_0 \setminus K_0, K_0)}{2}\\
            & = g(K_0) + \alpha/2\\
            & = g(K_1) + \alpha/2.
        \end{split}
    \end{equation*}

    It follows that $m_{K_0} = m_{K_1}$.
    We also have inequalities
        \[2^{n-1} \leq \rank \hatfl(L_0, m_{K_0}) = \rank \hatfl(L_0, m_{K_1}) \leq \rank \hatfl(L_1, m_{K_1}) = 2^{n-1},\]
    where the first inequality is Proposition \ref{prop:braidaxisdetect} (b). and the second inequality is Theorem \ref{thm:linkfloermonotonicity}.
    Therefore, $\rank \hatfl(L_0, m_{K_0}) = 2^{n-1}$, and $L_0 \setminus K_0$ is a braid closure in the complement of $K_0$ by Proposition \ref{prop:braidaxisdetect} (b).
    Moreover, the index of $L_0 \setminus K_0$ as a braid closure in the complement of $K_0$ is precisely $\alpha$ from Proposition \ref{prop:braidaxisdetect} (c). and our calculation of $m_{K_0}$.
\end{proof}

Our second major component relies on observations of Gordon, and in particular his strategy of using the transfinitely nilpotent condition to restrict ribbon concordance.

\begin{definition}[Page 157 of \cite{Gordon-ribcon}]
    \label{def:transfinitelynilpotent}
    A group $G$ is called \textit{transfinitely nilpotent} if the intersection of the pieces 
        \[[G, G], [[G, G], G], ...\]
    of the lower central series of $G$ terminates to $\{1\}$ at some ordinal.
\end{definition}

Gordon observes that a free group is transfinitely nilpotent \cite{Gordon-ribcon}.
This gives us all of the tools we need to prove the main theorem.

\begin{proof}[Proof of Theorem \ref{thm:braidclosureunionbraidaxisisminimal}]
    $(\Longrightarrow)$
    We start with a link $L_1$ and a fibered component $K_1$ of $L_1$ such that $K_1$ is ribbon concordance minimal and $L_1 \setminus K_1$ is a braid closure of index $\alpha$ in the complement of $K_1$.
    We will henceforth write $L_1 = \hat{\beta}_1 \cup K_1$, with $\hat{\beta}_1$ being our braid closure.

    Suppose that $L_0 \leq L_1$, with the component $K_0$ of $L_0$ being sent to $K_1$ under our ribbon concordance.
    By Lemma \ref{lma:braidclosurepreservation}, we can write $L_0 = \hat{\beta}_0 \cup K_0$, where $\hat{\beta}_0$ is a braid closure of index $\alpha$ in the complement of the fibered link $K_0$.
    We then obtain a ribbon concordance $R: \hat{\beta}_0 \cup K_0 \to \hat{\beta}_1 \cup K_1$ sending $\hat{\beta}_0$ to $\hat{\beta}_1$ and $K_0$ to $K_1$.

    Using Gordon's notation, we set $Y = (S^3 \times [0,1]) \setminus \nu R$ and $X_{0, 1} = S^3 \setminus \nu(L_{0,1})$.
    Recall that the exterior of a link concordance is a ribbon $\Z$-homology cobordism, meaning in particular that we have isomorphisms $H_1(X_0; \Z) \cong H_1(Y; \Z) \cong H_1(X_1; \Z) \cong \Z^n$.
    Under this correspondence, the element $(1, 0, ..., 0) \in H_1(X_0; \Z)$ corresponding to the meridian of $K_0$ is sent to the element $(1, 0, ..., 0) \in H_1(Y; \Z)$ corresponding to the meridian of the component of $R$ between $K_0$ and $K_1$, which is subsequently sent to the $(1, 0, ..., 0) \in H_1(X_1; \Z)$ corresponding to the meridian of $K_1$.
    Let $\widetilde{Y}$ be the infinite cyclic cover of $Y$ determined by the homomorphism $H_1(Y; \Z) \to \Z$ sending $(1, 0, ..., 0)$ to $1$ and all other generators to $0$.
    Then, the correspondences we noted above give us $\partial \widetilde{X}_0, \partial \widetilde{X}_1 \subset \partial \widetilde{Y}$, where $\widetilde{X}_{0,1}$ is the infinite cyclic cover of $X_{0,1}$ constructed by the same homomorphism.  
    
    As we have Lemma \ref{lma:pi1injsurj}, the first paragraph of Lemma 3.2 in \cite{Gordon-ribcon} carries through exactly to give $H_2(\widetilde{Y}; \Z) = 0$.
    This is also Step 2 of the proof of Theorem 1.9 in \cite{Dunkerley-linkribcon}, but we caution the reader that $Y$'s represent $3$-manifolds and $X$'s represent $4$-manifolds in Dunkerley's notation.

    As $\hat{\beta}_0, \hat{\beta}_1$ have the same index, we know that $\widetilde{X}_0, \widetilde{X}_1$ are both homeomorphic to $\Sigma_\alpha \times \R$, where here $\Sigma_\alpha$ denotes a surface of genus $g(K) = g(K_0)$ with one boundary component and $\alpha$ punctures.
    In particular, we have an equality $\dim H_1(\widetilde{X}_0; \Q) = \dim H_1(\widetilde{X}_1; \Q)$.
    Since $(\iota_1)_*: \pi_1(X_1) \rightarrow \pi_1(Y)$ is surjective by Lemma \ref{lma:pi1injsurj}, we know that there is a surjection $H_1(\widetilde{X}_1; \Q) \rightarrow H_1(\widetilde{Y}; \Q)$, and therefore a surjection $H_1(\partial \widetilde{Y}; \Q) \rightarrow H_1(\widetilde{Y}; \Q)$.
    With access to a version of Poincar\'e duality for infinite cyclic covers (see Milnor \cite{Milnor-infcyccov}), the induced map $i^* : H^1(\widetilde{Y}; \Q) \to H^1(\partial\widetilde{Y}; \Q)$ is dual to the boundary map $\delta: H^1(\partial \widetilde{Y}; \Q) \to H^2(\widetilde{Y}, \partial \widetilde{Y}; \Q)$ with respect to Milnor's pairing.
    We then obtain a version of the ``half lives, half dies'' theorem, which is to say that there are equalities
        \[\frac{1}{2}\dim H_1(\partial \widetilde{Y}; \Q) = \dim \ker(H_1(\partial \widetilde{Y}; \Q) \rightarrow H_1(\widetilde{Y}; \Q)) = \dim H_1(\partial \widetilde{Y}; \Q) - \dim H_1(\widetilde{Y}; \Q).\]
    It follows that

    \begin{equation*}
        \begin{split}
            \dim H_1(\widetilde{Y}, \Q) & = \frac{1}{2}\dim H_1(\partial \widetilde{Y}; \Q)\\
            & = \frac{\dim H_1(\widetilde{X}_1; \Q) + \dim H_1(\widetilde{X}_0; \Q)}{2}\\
            & = \dim H_1(\widetilde{X}_1; \Q).
        \end{split}
    \end{equation*}

    Thus the surjection $H_1(\widetilde{X}_1; \Q) \rightarrow H_1(\widetilde{Y}; \Q)$ has to be an isomorphism.
    We know that $\pi_1(\widetilde{X}_1) \cong \pi_1(\Sigma_\alpha)$ is free, and is therefore transfinitely nilpotent.
    The third paragraph of Lemma 3.2 in \cite{Gordon-ribcon} tells us that the inclusion $(\widetilde{\iota}_1)_*: \pi_1(\widetilde{X}_1) \rightarrow \pi_1(\widetilde{Y})$ is injective, while its surjectivity follows from Lemma \ref{lma:pi1injsurj}. 
    As Dunkerley notes, the fact that $(\iota_1)_*$ is an isomorphism follows from the five lemma:

    \begin{center}
        \begin{tikzcd}
1 \arrow[r] \arrow[d] & \pi_1(\widetilde{X}_1) \arrow[r, hook] \arrow[d, "(\widetilde{\iota}_1)_*"] & \pi_1(X) \arrow[r, two heads] \arrow[d, "(\iota_1)_*"] & \mathbb{Z} \arrow[r] \arrow[d, "\cong"] & 1 \arrow[d] \\
1 \arrow[r]           & \pi_1(\widetilde{Y}) \arrow[r, hook]                                        & \pi_1(Y) \arrow[r, two heads]                          & \mathbb{Z} \arrow[r]                    & 1          
        \end{tikzcd}
    \end{center}    

    By Lemma \ref{lma:pi1injsurj}, the map $(\iota_0)_*: \pi_1(X_0) \to \pi_1(Y) \cong \pi_1(X_1)$ is an injection.
    Dunkerley proves in careful detail in Case i of Theorem 1.1 of \cite{Dunkerley-linkribcon} that this injection respects the peripheral structure of $L_0, L_1$.
    The main idea is featured in the fourth paragraph of Lemma 3.2 in \cite{Gordon-ribcon}: 
    There is a homeomorphism $\partial \nu R \cong \sqcup_{i=1}^n T^2 \times I$ providing an identification $g$ between $\partial X_0$ and $\partial X_1$, whereby $(\iota_0)_*|_{\partial X_0} = g_*$.
    Moreover, Dunkerley's argument demonstrates that $(\iota_0)_*$ needs to be an isomorphism in the case that there is an identification $\pi_1(Y) \cong \pi_1(X_1)$.
    
    Since $X_0, X_1$ are both complements of non-split links and are thus Haken, we can use Corollary 6.4 of \cite{Waldhausen-irred3man} to conclude that the map $(\iota_0)_* : \pi_1(X_0) \to \pi_1(X_1)$ is induced by a covering map $f: X_0 \to X_1$.
    Then $f_*$ is an isomorphism, meaning our covering is actually a homeomorphism respecting the peripheral structure of $\partial X_0, \partial X_1$.
    It follows that $L_0 \cong L_1$, as desired.

    $(\Longleftarrow)$
    We will show that if a fibered knot $K$ is not ribbon concordance minimal, then there exists a braid closure $\hat{\beta}$ in $S^3 \setminus K$ (in fact, the proof will produce many) such that $K \cup \hat{\beta}$ is not ribbon concordance minimal.
    
    Let $R: K' \to K$ be a ribbon concordance such that $K' \neq K$, and set $R_t = R \cap (S^3 \times \{t\})$.
    Choose a path $\tau: [0,1] \to S^3 \times [0,1]$ on $R$ such that $\tau(t) \in R_t$, and $\tau$ misses all critical points of $R$.
    Let $D_t$ be a family of small disks such that $D_t \subset S^3 \times \{t\}$ intersects $R_t$ in a single point at $\tau(t)$.
    We can then define the annulus $S = \bigcup_t \partial D_t$, whereby we obtain a ribbon concordance $R \cup S : K' \cup \partial D_0 \to K \cup \partial D_1$.

    We know that there is a diffeomorphism $f: S^1 \times D^2 \to \nu \partial D_1$.
    Let $\hat{\beta}$ be any braid closure in $S^1 \times D^2$, and identify $\hat{\beta}$ with its image under $f$.
    Then $\hat{\beta}$ is a braid closure in $S^3 \setminus K$ (it hugs a meridian for $K$), and so we set $L = K \cup \hat{\beta}$.
    We actually have a family of diffeomorphisms $f_t: S^1 \times D^2 \to \nu \partial D_t$, and so we get a new annulus $S'$ by taking 
        \[S' = \bigcup_t f_t(\hat{\beta}).\]
    When all is said and done, we have a ribbon concordance $R \cup S' : K' \cup \hat{\beta} \to K \cup \hat{\beta}$ (again identifying $\hat{\beta}$ with its image under the appropriate $f_t$).

    Suppose for contradiction that $K' \cup \hat{\beta} = K \cup \hat{\beta}$.
    Then the injection of Theorem \ref{thm:linkfloermonotonicity} induces an isomorphism $\hatfl(K' \cup \hat{\beta}) \cong \hatfl(K \cup \hat{\beta})$.
    Since $K' \leq K$ and $K$ is fibered, $K'$ is fibered by combined results of \cite{Silver-knotlikegroupsribcon} and \cite{Kochloukova-knotlikegroups}.
    We know that $\partial D_0$ is a meridian for $K'$, so $\hat{\beta}$ is braided in the complement of $K'$ as well.
    Applying Proposition \ref{prop:braidaxisdetect} (c). and the argument of Lemma \ref{lma:braidclosurepreservation}, it follows that $g(K') = g(K)$ by comparing maximally supported Alexander gradings associated to $K', K$.
    However, this means that $K' = K$ by Lemma 3.4 of \cite{Gordon-ribcon}, a contradiction.
    Thus, $K' \cup \hat{\beta} \neq K \cup \hat{\beta}$, and $K \cup \hat{\beta}$ is not ribbon concordance minimal.
\end{proof}

Let $Y_-, Y_+$ be compact $3$-manifolds with empty or toroidal boundary.
A recent result of Sun \cite{Sun-fiberedclasses} shows that if there is a ribbon $\Q$-homology cobordism (see \cite{DaLidVVShea-ribhomcob}) from $Y_-$ to $Y_+$, then the fibered classes of $Y_+$ are identified with a subset of the fibered classes of $Y_-$.
Qiuyu Ren informed the author that this result, together with a generalization of Theorem 1.7 of \cite{AgolRen-fiberedribcon} to the case of fibered links, should offer a road to an alternate proof of the forward direction of Theorem \ref{thm:braidclosureunionbraidaxisisminimal}.
Given a link $L$ as in the statement of Theorem \ref{thm:braidclosureunionbraidaxisisminimal}, a fiber of $L$ is a punctured genus $g(K)$ surface.
If $L' \leq L$, then $L'$ is fibered by Corollary 1.2 of \cite{Sun-fiberedclasses}.
Moreover, the monodromy of $L$ compresses to the monodromy of $L'$.
When $L' \neq L$, filling in the punctures of our surface gives a nontrivial compression of the monodromy of $K$.
We plan to explore this line of reasoning in future work.

\section{Acknowledgments}

The author would like to thank Ian Agol and Melissa Zhang for their guidance and for many inspiring conversations on ribbon concordance.
Special thanks also to Fraser Binns for making the author aware of stronger versions of braid axis detection theorems.
We are also grateful to Gary Dunkerley and Qiuyu Ren for insightful correspondence, and to Wren Burrill for helpful feedback.

\bibliographystyle{alpha}
\bibliography{main}

\end{document}